\documentclass[12pt,reqno]{amsart}
\usepackage[utf8]{inputenc}
\usepackage{graphicx}
\usepackage{amsmath,amssymb}
\usepackage{hyperref}
\usepackage{color}
\usepackage{bm}

\renewcommand{\div}{\rm div}

\numberwithin{equation}{section}
\newtheorem{theorem}{Theorem}[section]
\newtheorem{definition}[theorem]{Definition}

\newtheorem{remark}[theorem]{Remark}
\newtheorem{lemma}[theorem]{Lemma}
\newtheorem{proposition}[theorem]{Proposition}
\newtheorem{thma}{Theorem}

\def\O{\Omega}
\def\C{\mathcal{C}}

\def\dist{{\rm dist}} 

\def\div{{\rm div}\,}
\def\supp{{\rm supp}}
\def\dx{{\,\rm d}x}
\def\dy{{\,\rm d}y}

\def\v{{\bf v}}

\def\length{{\rm length}}
\def\dim{{\rm dim}}

\newcommand{\R}{{\mathbb R}}
\newcommand{\N}{{\mathbb N}}

\DeclareMathOperator*{\esssup}{ess\,sup}
\DeclareMathOperator*{\essinf}{ess\,inf}

\title[Some inequalities on weighted Sobolev spaces]{Some inequalities on weighted Sobolev spaces, distance weights and the Assouad dimension}

\date{}

\author[F. L\'opez-Garc\'\i a]{Fernando L\'opez-Garc\'\i a}
\address{Department of Mathematics and Statistics\\ California State Polytechnic University Pomona\\
3801 West Temple Avenue\\ Pomona, CA 91768, US} 
\email{fal@cpp.edu}

\author[I. Ojea]{Ignacio Ojea}
\address{Departamento de Matemática, Facultad de Ciencias Exactas y Naturales, Universidad de Buenos Aires e IMAS - CONICET} 
\email{iojea@dm.uba.ar}

\begin{document}

\begin{abstract}
We consider certain inequalities and a related result on weighted Sobolev spaces on bounded John domains in $\R^n$. Namely, we study the existence of a right inverse for the divergence operator, along with the corresponding a priori estimate, the improved and the fractional Poincar\'e inequalities, the Korn inequality and the local Fefferman-Stein inequality. 
 All these results are obtained on weighted Sobolev spaces, where the weight is a power of the distance to the boundary. In all cases the exponent of the weight $d(\cdot,\partial\Omega)^{\beta p}$ is only required to satisfy the restriction: $\beta p>-(n-\dim_A(\partial\Omega))$, where $p$ is the exponent of the Sobolev space and $\dim_A(\partial\Omega)$ is the Assouad dimension of the boundary of the domain. According to our best knowledge, this condition is less restrictive than the ones in the literature. 
\end{abstract}



\thanks{The second author is supported by ANPCyT under grant PICT 2018 - 3017, by CONICET under grant PIP112201130100184CO and by Universidad de Buenos Aires under grant 20020170100056BA}

\maketitle

\section{Introduction}
In this paper, we consider certain inequalities and a related result on weighted Sobolev spaces on bounded John domains $\Omega$ in $\R^n$, where the weight is a power of the distance to the boundary $\partial\Omega$. Indeed, we study the validity of the improved and fractional Poincar\'e inequalities, the Korn inequality and the local Fefferman-Stein inequality, and the existence of a solution to the divergence equation on weighted Sobolev spaces $W^{1,p}(\Omega,d^{\beta p})$, where $1<p<\infty$ and $d$ is the distance to $\partial\Omega$. All these inequalities have been widely studied by several authors under different assumptions on the domain and weights. It is well-known that the validity of these weighted inequalities is strongly related to the geometry of the domain and its boundary. However, even in the important but restricted case of weights $d(\cdot,\partial\Omega)^{\beta p}$ given by powers of the distance to the boundary, a characterization of the exponent $\beta p$ for which the inequalities hold in terms of geometric properties of the domain is yet to be obtained. In this direction, the main goal of the present work is to show a relation between the exponents of the weights and the Assouad dimension of the boundary. In particular, we prove that the inequalities considered here hold under the sufficient condition: \[\beta p > -(n-\dim_A(\partial\Omega)),\]
where $\dim_A(\partial\Omega)$ is the Assouad dimension of the boundary. For smooth domains, with boundary given by a regular $n-1$ dimensional set, the restriction says $\beta p > -1$. However, John domains may have intricate boundaries, with Assouad dimension greater than $n-1$. For example, we may consider the planar domain enclosed by the Koch snowflake, which Assouad dimension equals $\frac{\ln(4)}{\ln(3)}$ (see \cite[Example 15]{DK} and references therein). In any case, our condition points to a relation between the possible exponents and the roughness of the boundary, expressed in terms of its Assouad dimension. Moreover, these results extend the range of exponents previously known in the literature.

The main argument used in this manuscript follows from a local-to-global technique that extends the validity of the inequalities from  Whitney cubes (local) to John domains (global). Indeed, our results are obtained as an application on John domains of the main theorem in \cite{LGO}, which states a sufficient condition for the validity of a certain weighted discrete Hardy-type inequality on trees, and a local-to-global argument based on a decomposition of functions introduced in \cite{L1}. Some variations of this local-to-global argument have been used under other geometric assumptions on the domain.  In \cite{L1}, it has been applied on H\"older-$\alpha$ and other irregular domains, whereas John domains were treated in \cite{L2}. In \cite{L1,L2}, this local-to-global argument is based on the continuity of a certain Hardy-type operator in $L^p(\Omega)$, which definition depends on the geometry of the domain. In \cite{BLV}, the Hardy operator in $L^p(\Omega)$ is replaced by the classical weighted discrete Hardy inequality in $\ell^p=\ell^p(\mathbb{N})$, where $\Omega$ is a particular cuspidal planar domain (a H\"older-$\alpha$ domain with a single cusp). This idea was extended to general bounded H\"older-$\alpha$ domains in \cite{LGO} by using a weighted discrete Hardy inequality on trees. In \cite{LGO}, there is a sufficient condition for H\"older-$\alpha$ domains that says $\beta p> -\alpha$, where the parameter $\alpha$ indicates how ``narrow" could be a cusp on the boundary of this type of domain. Observe that Lipschitz domains are John domains, but also H\"older-$\alpha$ domains with $\alpha=1$. In this case, the restriction on the exponent obtained here aligns with the one in \cite{LGO}.

Interestingly, the classical weighted discrete Hardy inequality also appears in a different local-to-global argument used in \cite{AO}, where general cuspidal domains with one singularity are considered. We also refer the reader to Section 4.5 in \cite{AD} for a description of the methods based on the classical weighted discrete Hardy inequality introduced in \cite{AO} and in \cite{BLV}. 
All the results mentioned in these paragraphs exhibit a connection between Hardy-type inequalities and the weighted inequalities treated in this manuscript.

%The paper is organized as follows: Section \ref{section notation} introduces two equivalent definitions of John domains, as well as Whitney decompositions. In Section \ref{section Assouad}, we provide a brief introduction into the Assouad dimension. In Section \ref{section decomposition}, we present in a condensed form the main results obtained in \cite{LGO} for bounded domains (see Theorem A), and prove that under an appropriate restriction on the weights it can be applied to bounded John domains.  Finally, in Section \ref{section applications}, we apply a local-to-global argument to several interesting problems. In particular, we prove the solvability of the divergence equation with its corresponding a priori estimate and the improved Poincar\'e, the fractional Poincar\'e, the Korn and the local Fefferman-Stein inequalities. For the local Fefferman-Stein inequality we provide an estimation on how the constant depends on the localization parameter. All these results are stated on weighted Sobolev spaces on bounded John domains, where the weights are powers of the distance to the boundary. 

\section{John domains and trees}
\label{section notation}

In this section, we recall two equivalent definitions of John domains, a standard definition and a more recent one based on Whitney cubes and trees. 

Throughout the paper $1< p,q<\infty$, with $\frac{1}{p}+\frac{1}{q}=1$, unless otherwise stated. 

A tree is a graph $(V,E)$, where $V$ is the set of vertices and $E$ the set of edges, satisfying that it is connected and has no cycles. A tree is said to be rooted if one vertex is designated as the root. In a rooted tree $(V,E)$, it is possible to define a {\it partial order} ``$\preceq$" in $V$ as follows: $s\preceq t$ if and only if the unique path connecting $t$ to the root $a$ passes through $s$. {\it The parent} $t_p$ of a vertex $t$ is the vertex connected to $t$ by an edge on the path to the root. It can be seen that each $t\in V$ different from the root has a unique parent, but several elements ({\it children}) in $V$ could have the same parent. Note that two vertices are connected by an edge ({\it adjacent vertices}) if one is the parent of the other one. For simplicity, we say that a set of indices $\Gamma$ has a tree structure if $\Gamma$ is the set of vertices of a rooted tree  $(\Gamma,E)$. Also, if the partial order ``$\preceq$" in $\Gamma$ is a total order (i.e. each element in $\Gamma$ has no more than one child), we say that $\Gamma$ is a {\it chain}, or has a chain structure.

We will work with discrete trees which are derived from a continuous setting. In this context, it is convenient to work with the set: \[\Gamma^* = \Gamma\setminus\{a\}.\]
Given a rooted tree $\Gamma$, we consider collections of real values indexed over $\Gamma^*$, named in this work as $\Gamma^*$-sequences or simply sequences. We define the space $\ell^p(\Gamma^*)$ of $\Gamma^*$-sequences ${\bm b} = \{b_t\}_{t\in\Gamma^*}$ that verifies that
\[\|{\bm b}\|_p= \left(\sum_{t\in\Gamma^*}b_t^p\right)^\frac{1}{p}<\infty.\]
We also define the \emph{path} $\mathcal{P}_t$ from the root $a$ to $t$: \label{PS}\[\mathcal{P}_t:=\{s\in \Gamma^*:a\prec s\preceq t\},\]
and the \emph{shadow} $\mathcal{S}_t$  of $t$: \[\mathcal{S}_t:=\{s\in\Gamma^*:\, s\succeq t\}.\]

For every pair of sequences ${\bm b}$ and ${\bm d}$, we say that ${\bm b}\sim{\bm d}$ if there is a positive real number $C$ such that \[\frac{1}{C} b_t\le d_t\le C b_t,\] for every $t\in\Gamma^*$. 
Similarly, we define $\Gamma$-sequences.

We recall the definition of a {\it bounded John domain}. 
A bounded domain $\Omega$ in $\R^n$ is a John domain with constants $a$ and $b$, $0<a\le b <\infty$, if
there is a point $x_0$ in $\Omega$ such that for each point $x$ in $\Omega$ there exists a rectifiable curve
$\gamma_x$ in $\Omega$, parametrized by its arc length written as  $\length (\gamma_x )$, such that
\[
\dist (\gamma_x (t),\partial \Omega )\geq \frac{a}{\length (\gamma_x)}t\quad \mbox{ for all } t\in [0,\length (\gamma_x)]
\]
and
\[
\length (\gamma_x)\le b.
\]
\begin{figure}[h]
    \centering    \includegraphics[scale=1.2]{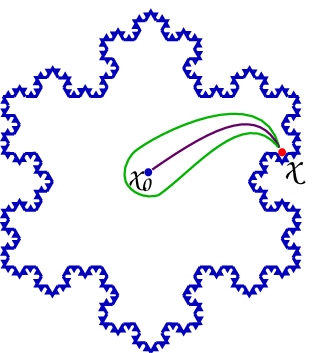}
    \caption{The Koch snowflake is an example of John domains. We show  a possible election of $x$ with its corresponding curve $\gamma_x$.}
    \label{koch}
\end{figure} 

Examples of John domains are convex domains, uniform domains, domains with slits, for example
$B^2(0,1)\backslash [0,1)$, and certain fractal domains such as the one with boundary given by the Koch snowflake (see Figure \ref{koch}), which is a particularly interesting example for our purposes since its boundary has a non-integer Assouad dimension.
On the other hand, the John property fails in domains with zero angle outward spikes. 
John domains were introduced by Fritz John in \cite{J} and named after John by O. Martio and J. Sarvas.

There are other equivalent definitions of John domains. In these notes, we are interested in a definition of the style of the Boman chain condition (see \cite{BKL}) in terms of Whitney decompositions and trees. This equivalent definition 
is  introduced in \cite{L2}. 

Let us recall that a Whitney decomposition of $\O$ is a collection $\{Q_t\}_{t\in\Gamma}$ of closed dyadic cubes whose interiors are pairwise disjoint, which satisfies
\begin{enumerate}
\item $\O=\bigcup_{t\in\Gamma}Q_t$,
\item $\text{diam}(Q_t) \leq d(Q_t,\partial\Omega) \leq 4\text{diam}(Q_t)$, \label{Whitney}
\item $\frac{1}{4}\text{diam}(Q_s)\leq \text{diam}(Q_t)\leq 4\text{diam}(Q_s)$, if $Q_s\cap Q_t\neq \emptyset$.
\end{enumerate}
Two different cubes $Q_s$ and $Q_t$ with $Q_s\cap Q_t\neq \emptyset$ are called {\it neighbors}. Notice that two neighbors may have an intersection with dimension less than $n-1$. For instance, they could be intersecting each other in a one-point set. We say that $Q_s$ and $Q_t$ are \mbox{$(n-1)$}-neighbors if $Q_s\cap Q_t$ is a $n-1$ dimensional face.
This kind of covering exists for any proper open set in $\R^n$ (see \cite{S} for details). Let us denote by $\ell_t=\ell(Q_t)$ the edge length of $Q_t$. Thus, if $Q_t^*$ is the cube with the same center as $Q_t$ and edge length $(1+\epsilon) \ell_t$ for a given $0<\epsilon<\frac{1}{4}$ then $Q_t^*$ touches $Q^*_s$ if and only if  $Q_t$ and $Q_s$ are neighbors. Thus, each expanded cube has no more than $12^n$ neighbors and $\sum_{t\in\Gamma}\chi_{Q_t^*}(x)\leq 12^n$.

Now we can state the following proposition proved in \cite{L2}, that establishes an equivalent definition for bounded John domains in terms of trees.

\begin{proposition}\label{tcovJohn} A bounded domain $\O$ in $\R^n$ is a John domain if and only if given a Whitney decomposition $\{Q_t\}_{t\in\Gamma}$ of $\O$, there exists a tree structure for the set of indices $\Gamma$ and a constant $K>1$ 
that satisfy
\begin{align}\label{Boman tree}
Q_s\subseteq KQ_t,
\end{align}
for any $s,t\in\Gamma$ with $s\succeq t$. In other words, the shadow of 
$Q_t$ given by $W_t=\cup_{s\succeq t} Q_s$ is contained in $KQ_t$. Moreover, the intersection of the cubes associated to adjacent indices, $Q_t$ and $Q_{t_p}$, is an $n-1$ dimensional face of one of these  cubes.
\end{proposition}

\section{Assouad dimension}
\label{section Assouad}

This section is devoted to present a brief introduction of the Assouad dimension and other notions.

\begin{definition}
Given a set $E\subset \R^n$, and $r>0$, we denote $N_r(E)$ the least number of open balls of radius $r$ that are needed for covering $E$. The \emph{Assouad dimension} of $E$, denoted $\dim_A(E)$ is the infimal $\lambda\geq 0$ for which there is a positive constant $C$, such that for every $0<r<R$ the following estimate holds:
\[\sup_{x\in E} N_r(B_R(x)\cap E) \le C\bigg(\frac{R}{r}\bigg)^\lambda.\]
\end{definition}

There is an alternative definition of the Assouad dimension that imposes an upper bound on the scale $R$. We are interested in the Assouad dimension of a bounded set, namely: the boundary of a bounded domain $\Omega$. In that case, both definitions are equivalent. However, the equivalence does not hold for unbounded domains. Also, what we defined as the Assouad dimension is sometimes referred as the \emph{upper} Assouad dimension, with a corresponding dual \emph{lower} dimension. We follow the notation of  \cite{Fraser}, where we refer the reader for an extensive study of the Assouad dimension. Here, we only comment some basic examples and properties that are relevant for our purposes. 

When compared with other common notions of dimension, the Assouad dimension turns out to be the largest. Indeed, let us denote $\dim_H(E)$ the Hausdorff dimension of $E$. Also, we recall the definition of the \emph{lower} and \emph{upper box dimensions}:

\[\underline{\dim}_B(E) = \liminf_{r\to 0}\frac{\log N_r(E)}{-\log(r)} \quad \quad \overline{\dim}_B(E) = \limsup_{r\to 0}\frac{\log N_r(E)}{-\log(r)}.\]
When both limits coincide, their value is the \emph{box dimension} of $E$, denoted $\dim_B(E)$. It is important to notice that the box dimension only makes sense for bounded sets, since for an unbounded domain $E$, $N_r(E)=\infty$.

With this definitions, we have that for any bounded set $E$:
\[\dim_H(E)\le \underline{\dim}_B(E)\le \overline{\dim}_B\le \dim_A(E).\]

For compact Ahlfors regular sets, the four dimensions are equal (\cite[Section 6.4]{Fraser}. On the other hand, for the simple example:
\[E = \{0\}\cup\{1/n:\,n\in\N\},\]
we have that:
\[\dim_H(E) = 0, \quad \dim_B(E) = \frac{1}{2}, \quad \dim_A(E) = 1.\]

The proof of this fact can be found in \cite[Theorem 2.1.1]{Fraser} and it serves to illustrate some of the differences between these dimensions: on one end, we have the Hausdorff dimension, which is global in nature and is equal to zero for every countable set. On the other end, the Assouad dimension is essentially local, since it considers only a small part of the domain at a time. This leads the Assouad dimension to \emph{see} the accumulation point at the origin as if it were part of a continuous line. In the middle, the box dimension retains the global nature of the Hausdorff one, while also catches some of the local behaviour at the origin. It is interesting to remark that the upper box dimension can be defined as the Assouad dimension, but fixing the larger scale $R$ to $1$.

\section{Orthogonal decomposition of functions}
\label{section decomposition}

In this section, we present in a condensed form the main results obtained in \cite{LGO} based on decomposition of functions for bounded domains (see Theorem A), and prove that under an appropriate restriction on the weights it can be applied to bounded John domains.

Let $\Omega\subset\R^n$ be a bounded domain with $n\geq 2$. We refer by a {\it weight} $\omega:\Omega\to\R$ to a Lebesgue-measurable function, which is positive almost everywhere. Then, we define the weighted spaces $L^p(\Omega,\omega)$ as the space of Lebesgue-measurable functions $f:\Omega\to\R$ with finite norm 
\[\|f\|_{L^p(\Omega,\omega)}=\left(\int_\Omega |f(x)|^p \omega(x)\, {\rm d}x\right)^{1/p}.\]
 Henceforth, $d:\Omega\to \R$ will denote the distance to $\partial\Omega$.

\begin{definition}\label{def tree covering}
Let $\Omega\subset\R^n$ be a bounded domain. We say that an open covering $\{U_t\}_{t\in\Gamma}$ is a {\it tree-covering} of $\Omega$ if it also satisfies the properties: 
\begin{enumerate}
\item $\chi_\Omega(x)\leq \sum_{t\in\Gamma}\chi_{U_t}(x)\leq C_1 \chi_\Omega(x)$, for almost every $x\in\Omega$, where $C_1\geq 1$.
\item The set of indices $\Gamma$ has the structure of a rooted tree.
\item There is a collection $\{B_t\}_{t\neq a}$ of pairwise disjoint open sets with $B_t\subseteq U_t\cap U_{t_p}$. Also, there is a constant $C_2$ such that: $\frac{|U_t|}{|B_t|}\le C_2$ for every $t\in\Gamma$.
\end{enumerate}
\end{definition}

A particular case of tree-covering is given by a set of expanded Whitney cubes, as the one used in the following definition (see also \cite[Definition 3.1 and Example 3.3]{LGO}).

\begin{definition}\label{Definition decomposition} 
Assuming that $\Omega\subset\R^n$ is a bounded domain and $\{Q_t\}$ is a Whitney decomposition of $\Omega$, we define $U_t=\frac{17}{16}Q_t$ an expansion of $Q_t$. Let us denote by $\mathcal{C}$ the space of constant functions from $\R^n$ to $\R$. Thus, given $g\in L^1(\Omega)$ orthogonal to $\C$ (i.e., $\int g\, \varphi=0$ for all $\varphi\in\C$), we say that a collection of functions $\{g_t\}_{t\in\Gamma}$
  in $L^1(\Omega)$ is a $\C$-{\it orthogonal decomposition of $g$} subordinate to $\{U_t\}_{t\in\Gamma}$ if the following three properties are satisfied:
\begin{enumerate}
\item $g=\sum_{t\in \Gamma} g_t.$
\item $\supp (g_t)\subset U_t.$
\item $\int_{U_t} g_t=0$, for all $t\in\Gamma$.
\end{enumerate}
\end{definition}
We also refer to this collection of functions by a $\C$-{\it decomposition}.  
Notice that condition (3) is equivalent to the orthogonality to the space $\C$. Indeed, this condition can be replaced by $\int_{U_t} g_t(x)\varphi(x)\dx=0$, for all $\varphi\in\C$ and $t\in\Gamma$. 

Finally, in this manuscript, we consider weights that satisfy:
\begin{equation}\label{admissible}
\esssup_{x\in U_t} \omega(x) \le C_3\essinf_{x\in U_t} \omega(x) \quad\forall t\in\Gamma,
\end{equation}
for some positive constant $C_3$. Notice that the condition defined above depends on the open covering $\{U_t\}_{t\in\Gamma}$ of $\Omega$. Also, observe that every weight of the form $d(x)^\eta$ satisfies \eqref{admissible}  with respect to a tree-covering formed by expanded Whitney cubes. 

Theorem \ref{Decomp Thm} is a simple application of the following fundamental Theorem proved in \cite{LGO}:

\begin{thma}\label{Theorem A} Let $\Omega\subset\R^n$ be a bounded domain, with a tree-covering $\{U_t\}_{t\in\Gamma}$, and let $\nu,\omega:\O\to\R$ be  weights satisfying \eqref{admissible}, such that $L^q(\O,\omega^{-q})\hookrightarrow L^q(\O,\nu^{-q})$. We also assume that the following supremum is finite for some $\theta>1$:
\begin{equation}\label{suff cond}
    A_{tree}:=\sup_{t\in\Gamma^*}\left(\sum_{a\prec s\preceq t} |B_s|^{-q/p}\nu_s^{-q}\right)^\frac{1}{\theta q}\left(\sum_{s\succeq t}|B_s|\omega_s^{p} \bigg(\sum_{a\prec r\preceq s}|B_r|^{-q/p}\nu_r^{-q}\bigg)^{\frac{p}{q}(1-\frac{1}{\theta})}\right)^\frac{1}{p},
\end{equation}
where $\nu_t=\essinf_{x\in U_t} \nu(x)$ and $\omega_t=\essinf_{x\in U_t} \omega(x)$, for any $t\in \Gamma$. 
Then, given $g\in L^q(\Omega,\omega^{-q})$, with  $\int_\O g=0$, there exists a $\C$-decomposition $\{g_t\}_{t\in\Gamma}$ of $g$ such that 
\begin{equation}\label{Decomp estim}
\sum_{t\in\Gamma} \int_{U_t}|g_t(x)|^q \nu^{-q}(x) \dx \leq C A_{tree}^q \int_{\Omega} |g(x)|^q \omega^{-q}(x)\dx.
\end{equation}
Moreover, the constant $C$ in \eqref{Decomp estim} depends only on the exponent $q$,  the constants $C_1$ and $C_2$ of the tree-covering that controls, respectively, the overlap of the ${U_t}'s$ and the ratios $|U_t|/|B_t|$, and on the constant $C_3$ in \eqref{admissible}. 
\end{thma}
\begin{proof}
The result is just a re-statement of \cite[Theorem 4]{LGO}. There, the finiteness of \eqref{suff cond} is presented in a more abstract context, as a sufficient condition for the validity of certain Hardy-type inequality on trees (inequality (2.1) in \cite{LGO}). Here, we replace the discrete weights, denoted $u_t$ and $v_t$ in \cite{LGO}, by the corresponding discretization of $\nu$ and $\omega$, namely: $u_t = |B_t|^\frac{1}{p}\nu_t$ and $v_t=|B_t|^\frac{1}{p}\omega_t$. See \cite[Inequality (3.2) and Remark 5]{LGO}. 
\end{proof}

\begin{remark}
Notice that the finiteness of the supremum in \eqref{suff cond} implies that the weight $\omega(x)^p$ belongs to $L^1(\Omega)$ (see \cite[Remark 6]{LGO}).
\end{remark}

\begin{theorem}\label{Decomp Thm}  Let $\Omega\subset\R^n$ be a bounded John domain and $\{U_t\}_{t\in \Gamma}$ a collection of expanded Whitney cubes as in Definition \ref{Definition decomposition}. Also, let $\beta\in\R$ such that 
\begin{equation}\label{cond beta}
    \beta p>-(n-\dim_A(\partial\Omega)).
\end{equation} 
Then, for every $g\in L^q(\Omega,d^{-\beta q})$, such that $\int_\O g=0$,  there exists a $\C$-decomposition $\{g_t\}_{t\in\Gamma}$ of $g$ that satisfies the estimate
\begin{equation}\label{Decomp estim dist}
\sum_{t\in\Gamma} \int_{U_t}|g_t(x)|^q d^{-\beta q }(x) \dx \leq C A_{tree}^q \int_{\Omega} |g(x)|^q d^{-\beta q}(x)\dx.
\end{equation}
\end{theorem}
\begin{proof}
This result is a direct application of Theorem \ref{Theorem A} in the particular case of a bounded John domain with weights $\omega=\nu=d^{\beta}$. Indeed, from Proposition \ref{tcovJohn}, it follows that there is a tree structure of $\Gamma$ such that the covering $\{U_t\}_{t\in\Gamma}$ verifies \eqref{Boman tree}, for some $K>1$. Moreover, thanks to the properties of a Whitney decomposition, for each $t$ there is an open set $B_t\subset U_t\cap U_{t_p}$, such that $\frac{|U_t|}{|B_t|}\leq C_2$, where $C_2$ depends only on $n$. Finally, the weight $d^{-\beta p}$ satisfies \eqref{admissible}. Also, notice that $\essinf_{x\in U_t} d(x)\sim \ell(Q_t)$, for any $t\in\Gamma$. Thus, the rest of the proof is devoted to proving the finiteness of $A_{tree}$ in \eqref{suff cond} for  $\omega_t=\nu_t=\ell(Q_t)^{\beta}$ and $|B_t|=\ell(Q_t)^n$.

We say that a set of Whitney cubes $Q_{t_0},\dots,Q_{t_m}$ is a \emph{chain} if $t_{i-1}$ is the parent of $t_i$ along $\Gamma$, for every $i=1,\dots,m$. In order to study \eqref{suff cond}, we need to estimate the two quantities defined below: 

\begin{align*}
    \mathbb{P}_i(t) &= \#\{r\in\Gamma:\,r\preceq t,\;\ell(Q_r)=2^{-i}\}, \\
    \mathbb{W}_i(t) &= \#\{r\in\Gamma:\,r\succeq t,\;\ell(Q_r)=2^{-i}\}.
\end{align*}

$\mathbb{P}_i(t)$ is the number of cubes of a given size in the chain from $a$ to $t$, whereas $\mathbb{W}_i(t)$ is the number of cubes of a given size in the shadow of $t$.

$\mathbb{P}_i(t)$ can be easily estimated. Indeed, let us consider the set $\mathcal{P} = \{Q_r: r\in \Gamma, t\preceq r,\, \ell_r=2^{-i}\}$, and let $Q_{r_0}$ be the first cube there. Then, in view of Proposition \ref{tcovJohn}, $\bigcup_{Q\in\mathcal{P}}Q \subset KQ_{r_0}$. Now, since $\ell(KQ_{r_0})=K2^{-i}$ the number of cubes of edge length $2^{-i}$ that can be packed in $KQ_{r_0}$ is at most $\frac{K^n2^{-in}}{2^{-in}}= K^n$. Hence:

\[\mathbb{P}_i(t) \le K^n.\]

The estimation of $\mathbb{W}_i(t)$ is a little more complicated, so we defer it to  Lemma \ref{estimate Wi(t)} below. It states that if $\ell(Q_t)=2^{-k}$ then: 
\[\mathbb{W}_i(t)\le C2^{(i-k)\lambda},\]
for every $\lambda>\dim_A(\partial\Omega)$. The constant depends on $\lambda$ and may go to infinity as $\lambda$ approaches $\dim_A(\partial\Omega)$. 

We also need to control the size of the cubes in a given chain or a given shadow. In general, given the nature of a Whitney decomposition, it is reasonable to expect that the edge length of the cubes will decrease along $\Gamma$, as we move away from the root and towards the boundary of $\Omega$. This is exactly the case for the tree defined in \cite{LGO} for H\"older-$\alpha$ domains. 
 However, Proposition \ref{tcovJohn} does not give us such tight control over the structure of $\Gamma$. Hence, we need to estimate how small $\ell(Q_r)$ can be with respect to $\ell(Q_t)$ for any cube in the chain $\mathcal{C} = \{Q_r: r\preceq t\}$. Respectively, how big $\ell(Q_r)$ can be with respect to $\ell(Q_t)$ for any cube in the shadow $\mathcal{S} = \{Q_r: r\succeq t\}$. This, again, is done thanks to Proposition \ref{tcovJohn}. For simplicity, we assume that $K=2^M$, for some positive integer $M$. 
We begin by analyzing the chain $\mathcal{C}$. Let $Q_s$ be the smallest cube in $\mathcal{C}$, with $\ell_s = 2^{-j}$. Also, let $\ell(Q_t) = 2^{-k}$ be the edge length of $Q_t$. Then, since $Q_t\subset KQ_s$, we have that $2^{-k} \leq K2^{-j}$, which implies that $j\leq k+M$. On the other hand and with the same analysis, let $Q_s$ be a cube with the biggest size in $\mathcal{S}$, and let us denote $\ell(Q_s) = 2^{-j}$. Since $Q_s\subset KQ_t$, we have that $2^{-j}\le K2^{-k}$, which implies that $j\ge k-M$. 
With these estimates, we can complete the proof. Observe that \eqref{suff cond} involves three indices in $\Gamma$: $r,s$ and $t$. We take integers $i,j,k$ such that $\ell(Q_r)=2^{-i}$, $\ell(Q_s) = 2^{-j}$ and $\ell(Q_t) = 2^{-k}$. Since $|B_s|\sim |U_s|$, we have that $|B_s|\sim 2^{-jn}$. Hence, using that $p\beta > -(n-\dim_A(\partial\Omega))\geq -n$, we have: 
\begin{align*}
    \sum_{a\prec s\preceq t} |B_s|^{-q/p}d_s^{-\beta q} 
    &\le C \sum_{a\prec s\preceq t} 2^{jnq/p}2^{j\beta q} \\
    &\le C\sum_{j=0}^{k+M}\mathbb{P}_j(t) 2^{jq(n+\beta p)/p} \\
    &\le CK^n\sum_{j=0}^{k+M}2^{jq(n+\beta p)/p} \le C 2^{kq(n+\beta p )/p},
\end{align*}
where $C$ does not depend on $k$ (recall that $\ell(Q_t)=2^{-k}$).
This gives the bound for the first factor in $A_{tree}$:
\[\bigg(\sum_{a\prec s\preceq t} |B_s|^{-q/p}d_s^{-\beta q}\bigg)^\frac{1}{q\theta} \le C 2^{k(n+\beta p)\frac{1}{\theta p}}. \]
To estimate the second factor we have: 
\[\bigg(\sum_{a\prec r\preceq s} |B_r|^{-q/p}d_r^{-\beta q}\bigg)^{\frac{p}{q}(1-\frac{1}{\theta})} \le C 2^{j(n+\beta p)(1-\frac{1}{\theta})}. \]
We take  $\lambda>\dim_A(\partial\Omega)$ to be chosen later and continue with the estimate of the second factor:
\begin{align*}
    \sum_{s\succeq t}|B_s|d_s^{\beta p} &\bigg(\sum_{a\prec r\preceq s}|B_r|^{-q/p}d_r^{-\beta q}\bigg)^{\frac{p}{q}(1-\frac{1}{\theta})} \\
    &\le C\sum_{j=k-M}^\infty \mathbb{W}_j(t) 2^{-jn}2^{-j\beta p} 2^{j(n+\beta p)(1-\frac{1}{\theta})} \\
    &\le C\sum_{j=k-M}^\infty 2^{(j-k)\lambda}2^{-jn}2^{-j\beta p} 2^{j(n+\beta p)(1-\frac{1}{\theta})} \\
    &\le C2^{-k\lambda} \sum_{j=k-M}^\infty 2^{j[\lambda-n-\beta p+(n+\beta p)(1-\frac{1}{\theta})]} \\
    &\le C2^{-k\lambda} \sum_{j=k-M}^\infty
    2^{j[\lambda-(n+\beta p)\frac{1}{\theta}]}.
\end{align*}
From condition \eqref{cond beta}, we can choose $\lambda>\dim_A(\partial\Omega)$ and  $\theta>1$ such that the exponent $\lambda-(n+\beta p)\frac{1}{\theta}$ is negative. Thus,  
\begin{align*}
    \sum_{s\succeq t}|B_s|d_s^{\beta p} &\bigg(\sum_{a\prec r\preceq s}|B_r|^{-q/p}d_r^{-\beta q}\bigg)^{\frac{p}{q}(1-\frac{1}{\theta})} \\
    &\le C2^{-k\lambda}2^{(k-M)[\lambda-(n+\beta p)\frac{1}{\theta}]} \\
    &= C2^{-k(n+\beta p)\frac{1}{\theta}},
\end{align*}
where the constant $C$ does not depend on $k$. Finally, from the estimates studied above we deduce that the argument in the supremum is bounded by: 
\[C2^{k(\frac{n}{p}+\beta)\frac{1}{\theta}}2^{-k(n+\beta p)\frac{1}{\theta}\frac{1}{p}} = C,\]
and the supremum $A_{tree}$ is finite, which completes the proof. Observe that $A_{tree}$ depends on $\lambda$ and may go to infinity as $\lambda$ tends to $\dim_A(\partial\Omega)$. In particular this implies that $A_{tree}$ may go to infinity as $\beta$ approaches $-(n-\dim_A(\partial\Omega))/p$.   
\end{proof}

It only remains to show the following lemma on the number of cubes of a given size in the shadow of a fixed cube.

\begin{lemma}\label{estimate Wi(t)}
Suppose $\Omega\subset\R^n$ is a bounded John domain and $\{Q_t\}_{t\in \Gamma}$ is a collection of Whitney cubes that satisfies condition \eqref{Boman tree}. Also, let $Q_t$ be a cube with $\ell(Q_t) = 2^{-k}$ and 
\[\mathbb{W}_i(t) = \#\{r\in\Gamma:\,r\succeq t,\;\ell(Q_r)=2^{-i}\}.\]
Then, given $\lambda>\dim_A(\partial\Omega)$ we have that
\[\mathbb{W}_i(t)\le C \left(\frac{2^i}{2^k}\right)^{\lambda}\]
where the constant $C$ depends on $n$, $\lambda$, and $K$ in condition \eqref{Boman tree}.

\end{lemma}
\begin{proof}
Let $\lambda>\dim_A(\partial\Omega)$. 

Given a cube $Q_r$ with $r\succeq t$, let us denote $x_r\in\partial\Omega$ such that $d(Q_r,\partial\Omega)=d(Q_r,x_r)$.  
Thanks to the definition of a Whitney decomposition and Proposition \ref{Boman tree} we have that there is a constant $C_n$ depending only on $n$ such that the following properties hold:

\begin{itemize}
\item $d(x_t,Q_t)\le C_n \ell(Q_t)$ and {\rm diam}$(Q_t)\le C_n \ell(Q_t)$,
\item $d(Q_t,Q_r)\le C_n K\ell(Q_t)$,
\item $d(Q_r,x_r)\le C_n K\ell(Q_t)$ and {\rm diam}$(Q_r)\le C_n K\ell(Q_t)$.
\end{itemize}

Therefore, it is easy to see that there is a potentially new constant $C_n$, depending only on $n$, such that: 
\[d(x_t,Q_r)\le C_n K\ell(Q_t)
    \quad \text{and}\quad |x_t-x_r|\le C_n K \ell(Q_t),\; \forall r\succeq t.\]

We define $B_R(x_t)$ the ball centered at $x_t$ with radius $R=C_n K \ell(Q_t)$. Then, for every $r\succeq t$, we have that the ball  $B_R(x_t)$ contains the cube $Q_r$ and the point in $\partial\Omega$ where the distance from $Q_r$ to $\partial\Omega$ is attained. Thus, it is sufficient to estimate the number of Whitney cubes of edge length $2^{-i}$ contained in $B_R(x_t)$.

From the definition of Assouad dimension and using that $2^{-i}<R$, we have:
\[N_{2^{-i}}(\partial\Omega\cap B_R(x)) \le C \bigg(\frac{R}{2^{-i}}\bigg)^\lambda,\]
where $C$ depends on $\lambda$. 
Now, let $B_1,\dots,B_m$ be the covering of balls of radius $2^{-i}$ of $\partial\Omega\cap B_R(x)$, with $m = N_{2^{-i}}(\partial\Omega\cap B_R(x))$. Using again the properties of a Whitney decomposition, there is $C_n$ depending only on $n$, such that the set 
\[\bigcup_{i=1}^m C_n B_i\]
contains all the Whitney cubes of edge length $2^{-i}$ contained in $B_R(x_t)$. On the other hand, each expanded ball $C_n B_i$ can pack at most $C_n$ cubes of edge length $2^{-i}$. Thus, we have that:
\[\mathbb{W}_t(i)\le C_n m \le C_{n,\lambda}\bigg(\frac{R}{2^{-i}}\bigg)^\lambda\le C_{n,\lambda}K^\lambda\bigg(\frac{\ell(Q_t)}{2^{-i}}\bigg)^\lambda = C_{n,\lambda}K^\lambda\bigg(\frac{2^i}{2^k}\bigg)^\lambda, \] which completes the proof.

%Since $Q_t$ is a Whitney cube, we have that  $9\sqrt{n} Q_t\cap\partial\Omega\neq\emptyset$. Let us take  $x\in\partial\Omega\cap 9\sqrt{n}Q_t$. Thanks to Proposition \ref{Boman tree} we have that $Q_r\subset KQ_t$ for every $r\succeq t$. Thus, the expanded cube $9\sqrt{n}K Q_t$ contains both $x$ and $W_t$, the shadow of $Q_t$.  We consider the ball $B_{R_0}(x)$ centered at $x$ with radius $R_0=9nK\ell(Q_t)$, which contains $W_t$. Given $Q_r\in W_t$, let us denote $x_r\in 9\sqrt{n}Q_r\cap\partial\Omega$. Since $Q_r\subset B_{R_0}(x)$, we have that $|x_r-x|\le R_0+\sqrt{n}\ell(Q_r)+9\sqrt{n}\ell(Q_r)\le (9nK+\sqrt{n}K+9\sqrt{n}K)\ell(Q_t)$. Therefore, we have that both $Q_r$ and $x_r$ are in the ball $B_R(x)$ with $R=(9n+\sqrt{n}+9\sqrt{n})K\ell(Q_t)$. In other words, $B_R(x)$ contains not only every cube $Q_r$ in $W_t$, but also every point in $\partial\Omega$ where the distance from a cube in $W_t$ to $\partial\Omega$ is attained. We will estimate the number of Whitney cubes of edge length $2^{-i}$ in $B_R(x)$, which is greater than $\mathbb{W}_i(t)$. 
\end{proof}

\section{Applications}
\label{section applications}

In this section we present several results regarding different inequalities on John domains. The proofs for these results are essentially identical to the ones given in \cite{LGO} for H\"older-$\alpha$ domains. The only relevant difference is that in a H\"older-$\alpha$ domain the singularities of the boundary impose a shift in the weight, so the inequalities involve two different weights for the left and the right hand side, respectively. On the contrary, for John domains this shift is unnecessary, so the same weight appears on both sides of the inequalities. It is important to notice that here we considered the local Fefferman-Stein inequality that is not treated in \cite{LGO}. The scheme for the proofs is as follows: given a function $f$ such that $\int_\Omega f = 0$, we decompose $f$ according to Theorem \ref{Decomp Thm}. Then, we apply an unweighted version of the inequality on each element $U_t$ of the partition, and take advantage of the estimate \eqref{Decomp estim dist} for recovering a global norm. For doing this we rely heavily on two facts: first, $d(x)$ can be regarded as constant over each $U_t$. Second, since $\{U_t\}$ are cubes, we can control the constant involved in the unweighted inequality. The divergence problem is solved \emph{directly}, by applying the decomposition to the data $f$. For the other results a duality characterization of the norm on the left hand side is used, and the decomposition is applied to the function in the dual space of the one where the function involved in the inequality belongs. For applying this argument we need the following lemma. We refer the reader to \cite[Lemma 3]{LGO} for the proof.

\begin{lemma}\label{lemma density}
Let $V$ be the subspace of $L^q(\Omega,d^{-\beta q})$ given by:
 \begin{align*}V:=\Big\{&g(x) +\psi d^{\beta p}(x)\colon \\  & g(x)\in L^q(\Omega,d^{-\beta q}) \text{ and } \psi\in\R, \text{with }\, \overline{\supp(g)}\subset \Omega, \int_\Omega g=0, \Big\}.
\end{align*} 
Then, $V$ is dense in $L^q(\Omega,d^{-\beta q})$, and any $g+\psi d^{\beta p}\in V$ verifies that  
\[\|g \|_{L^q(\Omega,d^{-\beta q})} \le 2\|g+\psi d^{\beta p}\|_{L^q(\Omega,d^{-\beta q})}.\]
\end{lemma}

\subsection{The divergence equation}

Given a regular enough bounded domain $\Omega\subset \R^n$ and $1<q<\infty$, there is a constant $C$ such that for any $f$ with $\int_\Omega f = 0$, the equation ${\rm div}\,{\bf u} = f$ has a solution in $\Omega$ that verifies the boundary condition ${\bf u} =\bm{0}$ on $\partial\Omega$ and the unweighted estimate
\begin{equation}\label{unweighted div ineq}
    \|D {\bf u}\|_{L^{q}(\Omega)}\le C\|f\|_{L^q(\Omega)}.
\end{equation} 

This problem has been widely studied for different types of domains and function spaces; including weighted Sobolev spaces. The solvability of this problem on John domains was first proved in \cite{ADM}, and it was extended to weighted spaces, for weights in the Muckenhoupt class $A_q$, in \cite{DRS}. 

Here we study the divergence problem for power distance weights. This case was previously treated in \cite{L2}, assuming that the exponent $\beta$ is non-negative. We extend the range of exponents, proving that the result holds for $\beta$ satisfying \eqref{cond beta}.

\begin{theorem}\label{Divergence} Let $\Omega\subset\R^n$ be a bounded John domain, and $\beta$ satisfying \eqref{cond beta}. Given $f\in L^q_0(\Omega,d^{-\beta q})$, with $\int_\Omega f=0$, there exists a vector field ${\bf u}\in W^{1,q}_0(\Omega,d^{-\beta q})^n$,
solution of ${\rm div\,}{\bf u}=f$, that verifies the estimate
\begin{equation}\label{weighted div ineq}
\|D{\bf u}\|_{L^q(\Omega,d^{-\beta q})}\le C\|f\|_{L^q(\Omega,d^{-\beta q})}.
\end{equation}
\end{theorem}

\begin{proof}
 Let us consider the decomposition $\{f_t\}_{t\in\Gamma}$  of $f$ given by Theorem \ref{Decomp Thm}. For each $f_t$, we have a solution ${\bf u}_t$ of $\div {\bf u}_t = f_t$, supported on $U_t$, with the unweighted estimate: \[\|D {\bf u}_t\|_{L^q(U_t)}\le C \|f_t\|_{L^p(U_t)}.\]

In addition, a simple scaling argument shows that a uniform constant $C$ can be chosen for every cube, and hence for every $t\in\Gamma$. Now we define ${\bf u} = \sum_{t\in\Gamma} {\bf u}_t$, which satisfies that $\div {\bf u} = f$ in $\Omega$ and ${\bf u}=0$ on $\partial\Omega$. Moreover, we can take a constant  $d_t\sim d(U_t,\partial\Omega)$ for each $t$, and: 
 \begin{align*}
    \|D {\bf u}\|_{L^q(\Omega,d^{-\beta q})}^q &\le C \sum_{t\in\Gamma} \|D {\bf u}_t\|_{L^q(U_t,d^{-\beta q})}^q\\
    &\le C \sum_{t\in\Gamma}d_t^{-\beta q}\|D {\bf u}_t\|_{L^q(U_t)}^q \\
    &\le C\sum_{t\in\Gamma}d_t^{-\beta q}\|f_t\|_{L^q(U_t)}^q \\
    &\le C \sum_{t\in\Gamma}\int_{U_t} |f_t(x)|^q d(x)^{-\beta q}\dx \\
    &\le C \int_\Omega |f(x)|^q d(x)^{-\beta q} \dx = C\|f\|_{L^q(\Omega,d^{-\beta q})}^q, 
\end{align*} 
where in the last step we used \eqref{Decomp estim dist}. 
\end{proof}

\subsection{Poincar\'e-type inequalities}
We consider the improved Poincar\'e and the fractional Poincar\'e inequalities. 
Improved Poincar\'e inequalities have been largely studied in several contexts. In the unweighted case, the improved Poincar\'e inequality establishes that for every $f$ with vanishing mean value on $\Omega$:
\begin{equation}\label{unweighted imp Poincare}
\|f\|_{L^p(\Omega)}\le C \| d^\eta\nabla f\|_{L^p(\Omega)},
\end{equation}
where $d=d(x,\partial\Omega)$ and $\eta$ is some value between $0$ and $1$. The validity of inequality \eqref{unweighted imp Poincare} on John domains, with $\eta=1$ is a particular case of a more general result obtained in \cite{H2}. On the other hand, on singular domains more restrictions on $\eta$ are necessary. For example, on a H\"older-$\alpha$ domain, inequality \eqref{unweighted imp Poincare} holds for $\eta=\alpha$ (see, for example \cite{BoasStraube,DMRT}, and \cite{ADL,LGO} for weighted extensions). 

\begin{theorem}\label{improved Poincare}
Let $\Omega$ be a John domain, and $f\in L^p(\Omega,d^{\beta p})$ for some $\beta$ satisfying \eqref{cond beta}, such that $\int_\Omega f d^{\beta p}=0$. Then, there is a constant $C$ such that:
\[\|f\|_{L^p(\Omega,d^{\beta p})}\le C\|\nabla f\|_{L^p(\Omega,d^{(\beta+1)p})}.\]
\end{theorem}
\begin{proof}
Let us apply Lemma \ref{lemma density} for a dual characterization of the norm of $f$.  Thus, following the representation of the elements in $V$ by $h=g+d^{\beta p}\psi$, we have
 \begin{align*}
    \|f\|_{L^p(\Omega,d^{\beta p})} &=\sup_{h\in V:\|h\|_{L^q(\Omega,d^{-\beta q})}=1} \int_\Omega f h \\
    &= \sup_{h \in V:\|h\|_{L^q(\Omega,d^{-\beta q})}=1}\int_\Omega f(g+d^{\beta p}\psi)\\
    &=\sup_{h \in V:\|h\|_{L^q(\Omega,d^{-\beta q})}=1}\int_\Omega fg.
\end{align*} 
In the last step, we used that $\int_\Omega f d^{\beta p} = 0$ and $\psi$ is a constant. Since $g$ has vanishing mean value, we can apply the decomposition of Theorem \eqref{Decomp Thm}. Then,
\[\|f\|_{L^p(\Omega,d^{\beta p})}=\sup_{h\in V:\|h\|_{L^q(\Omega,d^{-\beta q})}=1}  \int_\Omega \sum_{t\in\Gamma}fg_t.\]

Now, the summation can be pulled out of the integral thanks to the fact that the support of $g$ is compact on $\Omega$ and consequently intersects a finite number of cubes $U_t$. We denote $f_{U_t} = \frac{1}{|U_t|}\int_{U_t} f$ and apply the $\mathcal{C}-$orthogonality of $g_t$ and the estimate \eqref{Decomp estim dist}, obtaining:
 \begin{align*}
    \int_\Omega fg &= \sum_{t\in\Gamma} \int_{U_t} f g_t = \sum_{t\in\Gamma} \int_{U_t} (f-f_{U_t})g_t\\ 
    &\le \sum_{t\in\Gamma} \|f-f_{U_t}\|_{L^p(U_t,d^{\beta p})}\|g_t\|_{L^q(U_t,d^{-\beta q})}  \\
    &\le \Big(\sum_{t\in\Gamma}\|f-f_{U_t}\|^p_{L^p(U_t,d^{\beta p})}\Big)^\frac{1}{p}\Big(\sum_{t\in\Gamma}\|g_t\|^q_{L^q(U_t,d^{- \beta q})}\Big)^\frac{1}{q}\\
    &\le C\Big(\sum_{t\in\Gamma}d_t^{\beta p} \|f-f_{U_t}\|^p_{L^p(U_t)}\Big)^\frac{1}{p}\|g\|_{L^q(\Omega,d^{-\beta q})}.
\end{align*} 

The estimate in Lemma \ref{lemma density} implies $\|g\|_{L^q(\Omega,d^{-\beta q})}\le 2$. On the other hand, in each $U_t$, we have inequality \eqref{unweighted imp Poincare} with $\eta=1$. Moreover, in $U_t$ the distance to $\partial{U_t}$ can be bounded above by the distance to $\partial\Omega$, obtaining: 
\[\|f-f_{U_t}\|^p_{L^p(U_t)}\le C\|\nabla f d(\cdot,\partial U_t)\|^p_{L^p(U_t)}\le Cd_t^p\|\nabla f\|^p_{L^p(U_t)}.\]
Notice that the previous inequality also follows from the classical Poincar\'e inequality with the well-known estimation of its constant in terms of the diameter.
Next, applying this inequality to the estimate above, we have:
 \begin{align*}
    \|f\|_{L^p(\Omega,d^{\beta p})} &\le C %
    \Big(\sum_{t\in\Gamma}d_t^{\beta p}d_t^p\|\nabla f\|_{L^p(U_t)}^p\Big)^\frac{1}{p} \\
    &\le C\Big(\sum_{t\in\Gamma}\|\nabla f\|_{L^p(U_t,d^{(\beta+1)p})}^p\Big)^\frac{1}{p}\\
    &= C \|\nabla f\|_{L^p(\Omega,d^{(\beta+1)p})}.
\end{align*} 
\end{proof}

Now, let us consider the fractional Poincar\'e inequality:
\begin{equation}\label{basic frac Poincare}
\inf_{c\in\R}\|u-c\|_{L^p(U)}\le C\left(\int_U\int_{U\cap B(x,\tau d(x))}\frac{|u(x)-u(y)|^p}{|x-y|^{n+sp}}{\rm d}y{\rm d}x\right)^\frac{1}{p},
\end{equation}
for $\tau\in(0,1)$. If we replace the right hand side by the classical seminorm of $W^{s,p}(U)$, for $0<s<1$, where the double integral is taken over $U\times U$, the result is known to hold on every bounded domain  (see \cite[Section 2]{DD}, \cite[Proposition 4.1]{HL}). In addition, in that case it is shown in \cite[Proposition 4.1]{HL} that the constant involved in the inequality is proportional to $diam(U)^{\frac{n}{p}+s}|U|^{-\frac{1}{p}}$. The stronger version \eqref{basic frac Poincare} is equivalent to the classical one on Lipschitz domains \cite[equation (13)]{Dyda}. We refer the reader to \cite{DD} where weighted improved versions of \eqref{basic frac Poincare} are proven in different kinds of non-Lipschitz domains, such as John, $s$-John and H\"older-$\alpha$ domains. In that paper the authors consider the general case with different exponents ($p$ and $q$) on the left and right hand side. In particular, for the weighted $L^p$ estimate, they obtain   the following inequality for John domains \cite[Theorem 3.1]{DD}: 
\begin{equation}\label{frac Poincare John}
    \inf_{c\in\mathbb{R}}\|u-c\|_{L^p(\Omega,d^{\beta p})}\le 
    C \left(\int_\Omega\int_{\Omega\cap B(x,\tau d(x))}\frac{|u(y)-u(x)|^p}{|y-x|^{n+sp}}\delta(x,y)^{(\beta+s)p} {\rm d}y{\rm d}x\right)^\frac{1}{p},
\end{equation}
where $\delta(x,y)=\min\{d(x),d(y)\}$, and with the restriction $\beta\geq 0$. Here we extend this result to negative values of $\beta$ that verify condition \eqref{cond beta}.

\begin{theorem}\label{Theorem frac Poincare}
Let $\Omega$ be a John domain, $u\in W^{s,p}(\Omega,d^{\beta p})$, with $s\in(0,1)$ and $\beta$ satisfying \eqref{cond beta}, and $\tau\in(0,1)$. Then, inequality \eqref{frac Poincare John} holds, with constant $C\tau^{s-n}$, where $C$ is independent of $\tau$.
\end{theorem}
\begin{proof}
For simplicity, we may assume that $\int_{\Omega} u d^{\beta p} = 0.$ 
As in Theorem \ref{improved Poincare}, we write the norm on the left hand side as a supremum on the functions $h=g+d^{\beta p}\psi \in V$:
\[\|u\|_{L^p(\Omega,d^{\beta p})} = \sup_{h:\|h\|_{L^q(\Omega,d^{-\beta q})}=1} \int_\Omega u h = \sup_{h:\|h\|_{L^q(\Omega,d^{-\beta q})}=1} \int_\Omega u g.\]
 Now, we apply the decomposition for the function $g$ and the estimate \eqref{Decomp estim dist}. Thus, for any set of constants $\{c_t\}_{t\in\Gamma}$, we have:
 \begin{align*}
\int_\Omega ug & =\sum_{t\in\Gamma} \int_{U_t}(u-c_t)g_t \\
&\le \sum_{t\in\Gamma}\|u-c_t\|_{L^p(U_t,d^{\beta p})}\|g_t\|_{L^q(U_t,d^{-\beta q)})} \\
&\le C \left(\sum_{t\in\Gamma}\|u-c_t\|_{L^p(U_t)}d_t^{\beta p}\right)^\frac{1}{p}.
\end{align*} 

For completing the proof, we invoke \cite[Proposition 4.2]{HL}, that states that for a cube $Q$ with side length $\ell(Q)$:
\[\inf_{c\in\R}\|u-c\|_{L^p(Q)}\le C_{n,p}\tau^{s-n}\ell(Q)^{s} 
\left(\int_{Q}\int_{Q\cap B(x,\tau \ell(Q))} \frac{|u(x)-u(y)|^p}{|x-y|^{n+sp}} \,{\rm d}y{\rm d}x\right)^\frac{1}{p},\]
 for any $\tau\in(0,1)$. We apply this inequality, taking into account that  $\ell(U_t)\sim d_t\sim d(x)\sim d(y)$ for every $x\in U_t$ and $y\in U_t$:  
 \begin{align*}
    \|u\|_{L^p(\Omega,d^{\beta p})} 
    &\le C_{n,p}\tau^{s-n}\left(\sum_{t\in\Gamma}\int_{U_t}\int_{U_t\cap B(x,\tau\ell(U_t))}\frac{|u(x)-u(y)|^p}{|x-y|^{n+sp}}\ell(U_t)^{sp}d_t^{\beta p} \,{\rm d}y{\rm d}x \right)^\frac{1}{p} \\
 &\le C_{n,p}\tau^{s-n}\left(\sum_{t\in\Gamma}\int_{U_t}\int_{U_t\cap B(x,\tau d(x))}\frac{|u(x)-u(y)|^p}{|x-y|^{n+sp}}d_t^{(\beta+s)p} \,{\rm d}y{\rm d}x\right)^\frac{1}{p} \\
 &\le C\tau^{s-n}\left(\sum_{t\in\Gamma}\int_{U_t}\int_{ B(x,\tau d(x))}\frac{|u(x)-u(y)|^p}{|x-y|^{n+sp}}\delta(x,y)^{(\beta+s)p}\,{\rm d}y{\rm d}x\right)^\frac{1}{p} \\
 &\le C\tau^{s-n}\left(\int_{\Omega}\int_{ B(x,\tau d(x))}\frac{|u(x)-u(y)|^p}{|x-y|^{n+sp}}\delta(x,y)^{(\beta+s)p} \,{\rm d}y{\rm d}x\right)^\frac{1}{p},
\end{align*} 
which completes the proof.
 \end{proof}

\subsection{The Korn inequality}
Given a vector field ${\bf u}\in W^{1,p}(U)^n$, we define its symmetric gradient as $\varepsilon({\bf u}) = \frac{D{\bf u}+D{\bf u}^t}{2}$. Korn's inequality establishes that
\begin{equation}\label{basic Korn}\|D{\bf u}\|_{L^p(U)}\le C\|\varepsilon({\bf u})\|_{L^p(U)},
\end{equation}
for every ${\bf u}$ such that $\int_U \frac{D{\bf u}-D{\bf u}^t}{2} = 0$, where the constant $C$ is independent of ${\bf u}$. It is proven in \cite{ADM} that this inequality holds on John domains. We prove the following weighted version of \eqref{basic Korn}.

\begin{theorem}\label{Korn}
Let $\Omega$ be a John domain, and ${\bf u}\in W^{1,p}(\Omega, d^{\beta p})^n$ with $\beta$ satisfying \eqref{cond beta}, such that $\int_\Omega \frac{D{\bf u}-D{\bf u}^t}{2}d^{\beta p}=0$ then,
\[\|D{\bf u}\|_{L^p(\Omega,d^{\beta p})}\le C \|\varepsilon({\bf u})\|_{L^p(\Omega,d^{\beta p})}.\]
\end{theorem}
\begin{proof}
Defining $\eta({\bf u}) = \frac{D{\bf u}-D{\bf u}^t}{2}$ we have that $D{\bf u} = \varepsilon({\bf u})+\eta({\bf u})$. Thus, it is enough to prove the estimate for the coordinates $\eta_{i,j}({\bf u})$ of the matrix $\eta({\bf u})$, that have vanishing weighted mean value. The estimate is obtained exactly as in Theorems \ref{improved Poincare} and \ref{Theorem frac Poincare}: the norm of $\eta_{i,j}({\bf u})$ is characterized by duality via Lemma \ref{lemma density}. The unweighted estimate \eqref{basic Korn} is known to hold for convex domains with a constant $C$ proportional to the ratio between the diameter of $U$ and the diameter of a maximal ball contained in $U$ (see \cite{D}). Hence, a uniform constant can be taken for the unweighted inequality on every cube $U_t$.
\end{proof}

\subsection{The local Fefferman–Stein inequality}

Given a domain $\Omega$ in $\R^n$, $\sigma\geq 1$, and $f$ in $L^1(\Omega)$, let us define {\it the restricted sharp maximal function} $M^\sharp_{{\rm res},Q,\sigma}f:\R^n\to [0,\infty]$ by 
\[M^\sharp_{{\rm res},\Omega,\sigma}f(x)=\sup_{Q\ni x\colon \sigma Q\subseteq \Omega} \dfrac{1}{|Q|}\int_Q |f(y)-f_Q|\dy\]
for all $x$ in $\Omega$, where the supremum is taken over all cubes $Q\subset\R^n$, with $\sigma Q\subseteq\Omega$,  that contains $x$. This maximal function is extended by zero outside of $\Omega$.

In this subsection, we work on a local version of the Fefferman-Stein inequality for functions with mean value zero which was introduced in \cite{DRS}:   
\begin{equation}\label{LFS inequality}
\|f\|_{L^p(\Omega)}\leq C \|M^\sharp_{{\rm res},U,\sigma}f\|_{L^p(\Omega)}.
\end{equation}
The main result of this last part is Theorem \ref{LFS on Omega}, which states the validity of inequality \eqref{LFS inequality} on bounded John domains and weighted spaces. The weights are powers of the distance to the boundary of the domain, where the exponents verify \eqref{Decomp Thm}. It also exhibits an estimation of the constant involved in the inequality in terms of $\sigma$. The proof is based on the local-to-global methods discussed above, 
and the local result stated in Lemma \ref{LFS on cubes 2}. Most of this subsection is devoted to proving Lemma \ref{LFS on cubes 2}, which follows from Lemma \ref{LFS on cubes} (stated in \cite{DRS} and based on \cite[Lemma 4]{I}) and other results.

\begin{lemma}\label{LFS on cubes}
Let $Q\subset\R^n$ be a cube, $1\leq p<\infty$, and $f\in L^1(Q)$, with vanishing mean value. If $M^\sharp_{{\rm res},Q,1}f\in L^p(Q)$, then $f\in L^p(Q)$ and 
\begin{equation}\label{LFS eq}
\|f\|_{L^p(Q)}\leq 10^{5np+n+1}\|M^\sharp_{{\rm res},Q,1}f\|_{L^p(Q)}.
\end{equation}
\end{lemma}

\begin{lemma}\label{union}
Let $\Omega_0,\Omega_1\subset\R^n$ be two bounded domains, with $|\Omega_0\cap \Omega_1|>0$, and $1< p<\infty$. If \eqref{LFS inequality} is valid on $\Omega_0$ and $\Omega_1$, with constant $C_0$ and $C_1$ respectively, then $\Omega:=\Omega_0\cup\Omega_1$ verifies \eqref{LFS inequality} with constant 
\begin{equation}\label{union constant}
C=8\max\{C_0,C_1\} \left(\frac{|\Omega_1|}{|\Omega_0\cap\Omega_1|}\right)^{1/p}.
\end{equation}
\end{lemma}

\begin{proof} Let $f\in L^1(\Omega)$ be a function with mean value zero such that $M^\sharp_{{\rm res},\Omega,1}f$ belongs to $L^p(\Omega)$. Then, since $M^\sharp_{{\rm res},\Omega_i,1}f(x)\leq M^\sharp_{{\rm res},\Omega,1}f(x)$, it follows that $M^\sharp_{{\rm res},\Omega_i,1}f$ belongs to $L^p(\Omega)$ for $i=0,1$. 

Let's use the dual characterization of the $L^p$-norm for functions with integral zero which states that
\begin{align*}
    \|f\|_{L^p(\Omega)} \leq 2 \sup \int_{\Omega} fg,
\end{align*}
where the supremum is taken over all $g$ in $L^q(\Omega)$, with $\|g\|_{L^q(\Omega)}\leq 1$ and $\int_{\Omega}g=0$. Thus, for any function $g$ that verifies that conditions, let's define the $\C$-orthogonal decomposition $\{g_0,g_1\}$ relative to $\{\Omega_0,\Omega_1\}$ in the following way:
\begin{align*}
    g_0(x)&=g(x)\chi_{\Omega_0\setminus\Omega_1}(x)+\frac{\chi_B(x)}{|B|}\int_{\Omega_1} g\\
    g_1(x)&=g(x)\chi_{\Omega_1}(x)-\frac{\chi_B(x)}{|B|}\int_{\Omega_1} g,
\end{align*}
where $B=\Omega_0\cap \Omega_1$. Using the H\"older inequality we conclude that
\begin{equation*}
\int_{\Omega}\left|\frac{\chi_B(x)}{|B|}\int_{\Omega_1} g\right|^q \leq \left(\frac{|\Omega_1|}{|B|}\right)^{q/p}\int_{\Omega_1}|g|^q.
\end{equation*}
Thus, after some straight forward estimations we have 
\begin{equation*}
\|g_0\|^q_{L^q(\Omega_0)}+\|g_1\|^q_{L^q(\Omega_1)} \leq 2^{q+1}\left(\frac{|\Omega_1|}{|B|}\right)^{q/p}\|g\|^q_{L^q(\Omega)}.
\end{equation*}
Finally, 
\begin{align*}
    \int_{\Omega} fg &= \int_{\Omega_0}fg_0+\int_{\Omega_1}fg_1 =\int_{\Omega_0}(f-f_{\Omega_0})g_0+\int_{\Omega_1}(f-f_{\Omega_1})g_1\\
    &\le \sum_{i=0}^1\|f-f_{\Omega_i}\|_{L^p(\Omega_i)}\|g_i\|_{L^q(\Omega_i)}  \\
    &\le \max\{C_0,C_1\}\sum_{i=0}^1\|M^{\sharp}_{res,\Omega_i,1}f\|_{L^p(\Omega_i)}\|g_i\|_{L^q(\Omega_i)} \\
    &\le \max\{C_0,C_1\}\bigg(\sum_{i=0}^1\|M^{\sharp}_{res,\Omega,1}f\|_{L^p(\Omega_i)}^p\bigg)^\frac{1}{p} \bigg(\sum_{i=0}^1\|g_i\|_{L^q(\Omega_i)}^q\bigg)^\frac{1}{q}\\
    &\le \max\{C_0,C_1\}2^{1/p}\|M^{\sharp}_{res,\Omega,1}f\|_{L^p(\Omega)} 2^{1+1/q} \left(\frac{|\Omega_1|}{|B|}\right)^{1/p},\\
\end{align*}
which completes the proof.
\end{proof}

Lemma \ref{LFS on cubes 2} below states an improved version of inequality \eqref{LFS eq} where the supremum that defines the restricted sharp maximal function considers a smaller collection of cubes. For its proof, we will use Theorem \ref{Theorem A} and the following equivalence shown in \cite[Theorem 3]{LGO} when the tree order in $\Gamma$ is a complete order (i.e. $\Gamma$ is a chain). 

\begin{thma}\label{Theorem B} If $\Gamma$ is a chain (i.e the partial order ``$\preceq$" in $\Gamma$ is a total order) then the supremum $A_{tree}$ in \eqref{suff cond} is finite for some $\theta>1$ if and only if $A_{chain}$, defined in \eqref{chain cond}, is finite:
\begin{equation}\label{chain cond}
A_{chain} = \sup_{t\in\Gamma^*} \left(\sum_{s\preceq t}|B_s|^{-q/p}\nu_s^{-q}\right)^\frac{1}{q}\left(\sum_{s\succeq t}|B_s|\omega_s^p\right)^\frac{1}{p}.
\end{equation}
Moreover,
\begin{equation*}
    A_{tree}\leq \theta^{1/p} A_{chain}.
\end{equation*}

\end{thma}

\begin{lemma}\label{LFS on cubes 2}
Let $\Omega\subset\R^n$ be a bounded domain and let $Q_0\subset\Omega$ be a cube that verifies that $\text{diam}(Q_0) \leq \text{dist}(Q_0,\partial\Omega)$. Then, for any $1<p<\infty$, $\sigma\geq 1$, and $f\in L^1(\Omega)$, with vanishing mean value on $Q_0$, there is a constant $C$ independent of $\sigma$ and $f$ such that 
\begin{equation*}
\|f\|_{L^p(Q_0)}\leq C \sigma^n\|M^\sharp_{{\rm res},\Omega,\sigma}f\|_{L^p(Q_0)}.
\end{equation*}
\end{lemma}

\begin{proof} To prove this result, we divide $Q_0$ into a regular partition of  sufficiently small cubes, where the ratio of the diameter of $Q_0$ over the diameter of the cubes in the partition is comparable to $\sigma$. Then, we use Lemma \ref{LFS on cubes} and Lemma \ref{union} for a local estimation with $\sigma=1$. Finally, we extend the local estimation to the original cube $Q_0$ by using Theorem \ref{Theorem A}. 

Let $m\in \N$ be such that $\sigma/3<m\leq 1+\sigma/3$ and let $\{Q_t\}_{t\in\Gamma}$ be the regular partition of $Q_0$ with $m^n$ closed cubes. The edge length of each cube is $\ell(Q_t)=\frac{L}{m}$, where $L=\ell(Q_0)$ is the edge length of $Q_0$. By applying a rigid motion, we can assume that $Q_0$ is the cube $(0,L)^n$. Thus, if $\Gamma$ is the index set defined by 
\[\Gamma=\{t\in\N^n \colon 1\leq t_i\leq m\},\]
then the partition $\{Q_t\}_{t\in\Gamma}$ of $Q_0$ can be written as
\[Q_t=\prod_{i=1}^n \dfrac{L}{n}[t_i-1,t_i].\]

Let us assign to $\Gamma$ a tree structure such that each index $t$ different from the root has exactly one parent, and $Q_t$ and $Q_{t_p}$ share a $n-1$ dimensional face. In Figure \ref{cube_tree}, we show an example of such a tree-covering of $Q_0$, for $n=2$ and $m=5$, where the arrows describe the way to descend to the root that is located in the position $(1,1)$. Finally, the tree-covering that we use in this lemma is $U_a=Q_a^\circ$ and $U_t=\left(Q_t\cup Q_{t_p}\right)^\circ$ for any  $t\in\Gamma\setminus \{a\}$. The circle in the upper index denotes the set interior.

\begin{figure}[h]
    \centering
    \includegraphics[width=4cm]{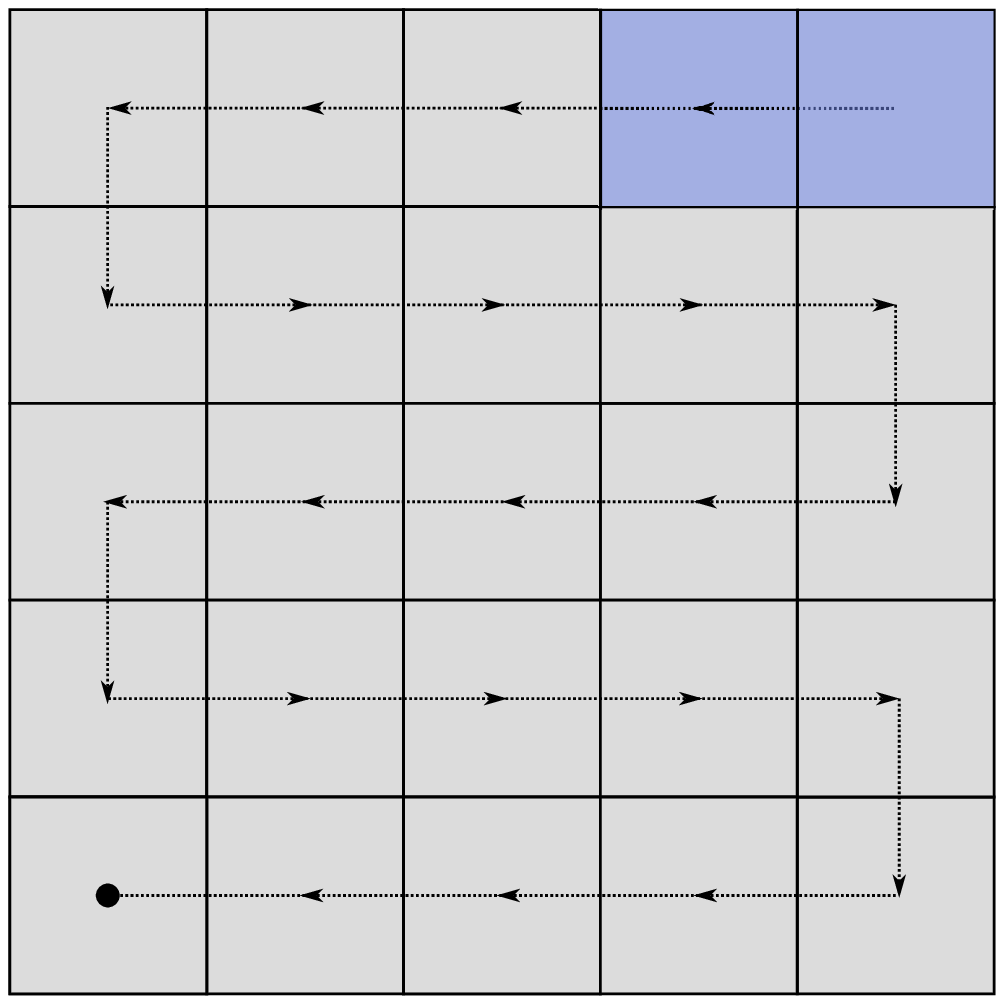}
    \caption{Tree-covering of a cube. The light-gray cubes are the $Q_t$ of the underlying partition. The dot indicates the root of the chain, whereas the dotted lines indicate the branches of the chain, with arrows pointing towards the root. The elements of the partition $U_t$ are the union of two consecutive cubes $Q_t$: an example is highlighted in blue.}
    \label{cube_tree}
\end{figure}

Observe that $U_t\cap U_{t_p}=Q_{t_p}^\circ$ for all $t\neq a$. 
Thus, let us define the collection $\{B_t\}_{t\neq a}$ of pairwise disjoint open cubes $B_t= Q_{t_p}^\circ$. Hence, $\{U_t\}_{t\in\Gamma}$ is a tree-covering of $Q_0$, with an overlapping constant $N=2$, and $|U_t|/|B_t|=2$ for all $t\neq a$.
This type of tree structure induces a total order ``$\preceq$" in $\Gamma$. Thus, the sufficient condition \eqref{suff cond} in Theorem \ref{Theorem A} can be replaced by the one that appears in Theorem \ref{Theorem B} for the unweighted case $\nu=\omega=1$. Now, 

\begin{align*}
A_{chain} &= \sup_{t\in\Gamma} \left(\sum_{a\prec s\preceq t}|B_t|^{-q/p}\nu_t^{-q}\right)^\frac{1}{q}\left(\sum_{s\succeq t}|B_s|\omega_t^p\right)^\frac{1}{p}\\
&\leq \sup_{t\in\Gamma} \left(\sum_{a\prec s\preceq t}L^{-qn/p}\,m^{qn/p}\right)^\frac{1}{q}\left(\sum_{s\succeq t}L^{-n}\, m^{-n}\right)^\frac{1}{p}\\
&= \sup_{t\in\Gamma} \left(\sum_{a\prec s\preceq t}1\right)^\frac{1}{q}\left(\sum_{s\succeq t}1\right)^\frac{1}{p}=  \left(\sum_{s\in \Gamma}1\right)^\frac{1}{q}\left(\sum_{s\in\Gamma}1\right)^\frac{1}{p}=m^n.
\end{align*}

Then, using Theorem \ref{Theorem A} and Theorem \ref{Theorem B}, we can conclude that there is an universal constant $C$ such that for any $g\in L^q(Q_0)$, with  $\int_{Q_0} g=0$, there exists a $\C$-decomposition $\{g_t\}_{t\in\Gamma}$ of $g$ such that 
\begin{equation*}
\sum_{t\in\Gamma} \int_{U_t}|g_t(x)|^q \dx \leq C A_{tree}^q \int_{Q_0} |g(x)|^q \dx \leq C m^{nq} \int_{Q_0} |g(x)|^q \dx.
\end{equation*}
Moreover, the constant $C$ in \eqref{Decomp estim} depends only on the exponent $q$,  the constants $C_1$ and $C_2$ of the tree-covering that controls, respectively, the overlap of the $U_t's$ and the ratios $|U_t|/|B_t|$, and on the admissibility constant $C_3$. Using that $m\leq 1+\sigma/3\leq \frac{4}{3}\sigma$, we conclude:
\begin{equation}\label{eq decomp sigma}
\sum_{t\in\Gamma} \int_{U_t}|g_t(x)|^q \dx \leq  C \sigma^{nq}\int_{Q_0} |g(x)|^q \dx.
\end{equation}

Finally, let us use the dual characterization of the $L^p$-norm, which implies for functions with integral zero that:
\begin{align*}
    \|f\|_{L^p(Q_0)} \leq 2 \sup \int_{Q_0} fg,
\end{align*}
where the supremum is taken over all $g$ in $L^q(Q_0)$, with $\|g\|_{L^q(Q_0)}\leq 1$ and $\int_{Q_0}g=0$. Notice that the local Fefferman-Stein inequality \eqref{LFS inequality} is valid on each $U_t$ with a constant that depends only on $n$ and $p$. This follows from Lemma \ref{LFS on cubes} and Lemma \ref{union}. Thus, using the $\C$-decomposition $\{g_t\}_{t\in\Gamma}$ of $g$ discussed above and Lemma \ref{LFS on cubes} we have: 
\begin{align*}
    \int_{Q_0} fg &=\sum_{t\in\Gamma} \int_{U_t}fg_t =\sum_{t\in\Gamma} \int_{U_t} (f-f_{U_t}) g_t \\
    &\le \sum_{t\in\Gamma}\|f-f_{U_t}\|_{L^p(U_t)}\|g_t\|_{L^q(U_t)}  \\
    &\le C\sum_{t\in\Gamma}\|M^{\sharp}_{res,U_t,1}f\|_{L^p(U_t)}\|g_t\|_{L^q(U_t)} \\
    &\le C\bigg(\sum_{t\in\Gamma}\|M^{\sharp}_{res,U_t,1}f\|_{L^p(U_t)}^p\bigg)^\frac{1}{p} \bigg(\sum_{t\in\Gamma}\|g_t\|_{L^q(U_t)}^q\bigg)^\frac{1}{q}. \\
%    &\le C \|M^{\sharp}_{res,U_t,1} f\|_{L^p(Q)}\bigg(\sum_{t\in\Gamma}\|g_t\|_{L^q(U_t)}^q\bigg)^\frac{1}{q}
\end{align*}

To estimate the second factor we use \eqref{eq decomp sigma}, and for the  first one we use that \[M^\sharp_{res,U_t,1} f(x)\le M^\sharp_{res,\Omega,\sigma}f(x). \]
Indeed, let $Q$ be a cube in $U_t\subseteq Q_0$. Then, the edge length $\ell(Q)\leq \ell(Q_t)=\ell(Q_0)/m$. Thus, 
\[\sigma Q\subseteq\frac{\sigma}{m} Q_0.\]
Now, since $\text{diam}(Q_0) \leq \text{dist}(Q_0,\partial\Omega)$ we have that $\frac{\sigma}{m} Q_0$ is included in $\Omega$ if $\frac{\sigma}{m} < 3$. This last condition follows from the selection of $m$ which verifies that $\sigma/3<m$. 
Thus, 
\begin{align*}
    \int_Q fg &\le C \sigma^n \|M^{\sharp}_{res,\Omega,\sigma} f\|_{L^p(Q)}\|g\|_{L^q(Q)},
\end{align*}
for any $g$ in $L^q(Q_0)$, with vanishing mean value on $Q_0$, which completes the proof.  
\end{proof}

\begin{theorem}\label{LFS on Omega}
Let $\Omega$ be a bounded John domain, $\sigma\geq 1$, and $f$ in  $L^p(\Omega,d^{\beta p})$ for some $\beta$ satisfying \eqref{cond beta}, such that $\int_\Omega f d^{\beta p}=0$. Then, there is a constant $C$ independent of $\sigma$ and $f$ such that
\[\|f\|_{L^p(\Omega,d^{\beta p})}\le C\sigma^n\| M^\sharp_{res,\Omega,\sigma} f\|_{L^p(\Omega,d^{\beta p})}.\]
\end{theorem}
\begin{proof}
The argument is very similar to the one used in the proof of the previous lemmas. We use the dual characterization of the norm of $f$ using $h=g+d^{\beta p}\psi \in V$:
\begin{align*}
    \|f\|_{L^p(\Omega,d^{\beta p})} &=\sup_{h\in V:\|h\|_{L^q(\Omega,d^{-\beta q})}} \int_\Omega fh = \sup_{h\in V:\|h\|_{L^q(\Omega,d^{-\beta q})}} \int_\Omega fg.
\end{align*}
    
We apply to $g$ the decomposition granted by Theorem \ref{Decomp Thm}, so we have a tree covering  $\{U_t\}_{t\in\Gamma}$ formed by cubes, and functions $g_t$ supported on $U_t$, with vanishing mean value, such that $g=\sum_{t\in\Gamma}g_t$ that verifies the estimate \eqref{Decomp estim dist}. Then, the H\"older inequality yields:

\begin{align*}
    \int_\Omega fg &=\sum_{t\in\Gamma} \int_{U_t} fg_t = \sum_{t\in\Gamma}\int_{U_t} (f-f_{U_t}) g_t 
    \le \sum_{t\in\Gamma}\|f-f_t\|_{L^p(U_t,d^{\beta p})}\|g_t\|_{L^q(U_t,d^{-\beta q})}.
\end{align*}

In the first factor of the summation, we can pull the weight out of the norm, using that for $x\in U_t$, $d(x)\sim d_t$. Then, we apply Lemma \ref{LFS on cubes 2} and finally the H\"older inequality and \eqref{Decomp estim dist}, leading us to:
\begin{align*}
    \int_\Omega fg &\le C\sigma^n\sum_{t\in\Gamma} d_t^\beta\|M^\sharp_{res,\Omega,\sigma} f\|_{L^p(U_t)}\|g_t\|_{L^q(U_t,d^{-\beta q})} \\
    &\le C\sigma^n\bigg(\sum_{t\in\Gamma}d_t^\beta\|M^\sharp_{res,\Omega,\sigma} f\|_{L^p(U_t)}^p\bigg)^\frac{1}{p}
    \bigg(\sum_{t\in\Gamma}\|g_t\|_{L^q(U_t,d^{-\beta q})}^q\bigg)^\frac{1}{q} \\
    &\le C\sigma^n\bigg(\sum_{t\in\Gamma}\|M^\sharp_{res,\Omega,\sigma} f\|_{L^p(U_t,d^{\beta p})}^p\bigg)^\frac{1}{p}
    \bigg(\sum_{t\in\Gamma}\|g_t\|_{L^q(U_t,d^{-\beta q})}^q\bigg)^\frac{1}{q} \\
    &\le C\sigma^n \|M^\sharp_{res,\Omega,\sigma} f\|_{L^p(\Omega,d^{\beta p})} \|g\|_{L^q(\Omega,d^{-\beta q})},
\end{align*}
which completes the proof, taking the supremum over all functions $h$ with unitary $L^q(\Omega,d^{-\beta q})$ norm. 
\end{proof}

This Theorem is similar to Theorem 5.23 in \cite{DRS}. There, the authors consider a general  weight $\omega$ in $A_p$. We study only the case $\omega=d^{\beta p}$, but for a range of values of $\beta$ that includes weights outside of $A_p$. Moreover, in \cite{DRS}, the parameter $\sigma$ is fixed and equal to a parameter of the \emph{emanating chain condition} that provides a characterization of John domains (see \cite[Definition 3.5]{DRS}). On the contrary, Lemma \ref{LFS on cubes 2} allows us to take an arbitrary value of $\sigma\ge 1$ and to show how the constant of the inequality depends on $\sigma$.

\section{Discussion}
It follows from \cite[Lemma 2.1.3]{LRZ} (see also \cite[page 42]{HP}) that $d(\cdot,\partial\Omega)^{\beta p}$ is integrable on a bounded domain $\Omega$ in $\R^n$ if $\beta p > -(n-\overline{\dim}_B(\partial\Omega))$, where $\overline{\dim}_B$ denotes the upper box dimension. So, since $\dim_A(E)\ge \overline{\dim}_B(E)$ for every bounded set $E$, a natural question that arises is if the weighted inequalities treated in this manuscript are still valid if we replace the Assouad dimension by the upper box dimension.  

The results studied in this manuscript, with the exception of the fractional Poincaré inequality, were also studied in \cite{DRS} for John domains. In this reference, the authors use a decomposition technique and the continuity of the Hardy-Littlewood maximal operator. It is well known that the Hardy-Littlewood maximal operator is continuous in $L^p(\Omega,\omega)$ if and only if the weight $\omega$ belongs to the Muckenhoupt class $A_p$. In this direction, it was recently shown in \cite{Dyda19} that a weight of the form $d(\cdot,E)^{\beta p}$ belongs to $A_p$ if and only if 
\begin{equation}\label{ApAssouad}
-(n-\dim_A(E))<\beta p <(n-\dim_A(E))(p-1),
\end{equation}
provided that $E$ in $\R^n$ is a porous set. Hence, it is interesting to note that \cite{DRS} combined with  \cite{Dyda19} might imply another argument to prove some of the inequalities studied here but under the condition  \eqref{ApAssouad}.

%It is relevant to mention that condition \eqref{ApAssouad} was previously obtained (as a sufficient condition to be in the $A_p$ class) in \cite[Lemma 3.3]{DLg} for the particular case in which $E$ in $\R^n$ is contained in an Ahlfors $m-$regular set. This was later extended in \cite[Theorem 7]{ACDT} to general metric measure spaces satisfying the Ahlfors condition.

It is relevant to mention that a similar condition to \eqref{ApAssouad} was previously obtained in \cite[Lemma 3.3]{DLg}, as a sufficient condition to be in the $A_p$ class. In this case, $\dim_A(E)$ is replaced by $\dim_A(\bar{E})$, where $\bar{E}$ is an Ahlfors regular set that contains $E$. This result was later studied in \cite[Theorem 7]{ACDT} on general metric measure spaces satisfying the Ahlfors condition.

Finally, it is known that inequalities like the ones treated in this manuscript hold for non-$A_p$ weights. This seems to point out certain limitations of the Hardy-Littlewood maximal operator as a tool for deriving this type of inequalities in a weighted setting. Here, we exploit the relation between these inequalities and Hardy-type inequalities following a discrete approach. 
However, for the sake of comparison, we can observe that our approach is a discrete analogous to the one followed in \cite{L2}, which relies on the continuity of the Hardy-type operator introduced in \cite[Definition 4.2]{L2}. It follows from its definition that this operator is upper-bounded on John domains by the following averaging operator:
\[Hg(x)=\dfrac{1}{|B(x,r)|}\int_{B(x,r)} |g|,\]
where $r=\alpha d(x,\partial\Omega)$ for a certain constant $\alpha>1$ that depends only on $\Omega$. For the integral in the definition of $H$, we extend the function $g$ by zero. Note that this averaging operator is bounded by the Hardy-Littlewood maximal operator. So, the continuity of the maximal operator implies the continuity of $H$. However, the reverse is not true. Thus, a better understanding of this averaging operator might help improve the results regarding these weighted inequalities.

\end{document}